\documentclass[11pt]{article}
\usepackage{amsmath,amsthm,amssymb,eqnarray}
\usepackage[latin1]{inputenc}
\usepackage{fourier} 
\usepackage[english]{babel} 
\usepackage{tikz}
\usetikzlibrary{decorations.pathreplacing}
\usetikzlibrary{decorations.markings}
\usepackage{mathtools} 
\usepackage{hyperref}
\usepackage[margin=1in]{geometry} 
 \usepackage{enumerate}
\usepackage{graphicx}
\hypersetup{
    colorlinks=true,
    linkcolor=blue,
    filecolor=magenta,      
    urlcolor=blue,
}
\usepackage{fancyvrb}
\usepackage{chngcntr}
\counterwithin{figure}{section}
\counterwithin{table}{section}
\counterwithin{equation}{section}
\usepackage{sectsty}
\usepackage{relsize}

\usepackage{color}

\theoremstyle{plain}

\newtheorem{theorem}{Theorem}[section]

\newtheorem{lemma}[theorem]{Lemma}

\newtheorem{proposition}[theorem]{Proposition}

\newtheorem{algorithm}[theorem]{Algorithm}

\newtheorem{remark}[theorem]{Remark}

\theoremstyle{remark}

\newcommand{\I}{\mathcal{I}}

\newcommand{\exd}{\gamma^*_e}

\DeclareMathOperator{\dist}{d}

\title{On exponential domination of the consecutive circulant graph} 
\author{Michael Dairyko$^1$ \and Michael Young$^1$}

\begin{document}
\maketitle
\footnotetext[1]{Department of Mathematics, Iowa State University, Ames, IA 50011, USA. (mdairyko, myoung) @iastate.edu}

\begin{abstract}
For a graph $G,$ we consider $D \subset V(G)$ to be a porous exponential dominating set if $1\le \sum_{d \in D}$ $\left( \tfrac{1}{2} \right)^{\dist(d,v) -1}$ for every $v \in V(G),$ where $\dist(d,v)$ denotes the length of the smallest $dv$ path. Similarly, $D \subset V(G)$ is a non-porous exponential dominating set is $1\le \sum_{d \in D} \left( \tfrac{1}{2} \right)^{\overline{\dist}(d,v) -1}$ for every $v \in V(G),$ where $\overline{\dist}(d,v)$ represents the length of the shortest $dv$ path with no internal vertices in $D.$ The porous and non-porous  exponential dominating number of $G,$ denoted $\exd(G)$ and $\gamma_e(G),$ are the minimum cardinality of a porous and non-porous  exponential dominating set, respectively. The consecutive circulant graph, $C_{n, [\ell]},$ is the set of $n$ vertices such that vertex $v$ is adjacent to $v \pm i \mod n$ for each $i \in [\ell].$ In this paper we show $\gamma_e(C_{n, [\ell]}) = \exd(C_{n, [\ell]}) = \left\lceil \tfrac{n}{3\ell +1} \right\rceil.$\\

\noindent \textbf{Key words:}  domination; generalized circulant graph; exponential domination; porous exponential domination\\

\noindent\textbf{AMS 2010 Subject Classification:} Primary 05C69; Secondary 05C12
\end{abstract}

\begin{section}{Introduction}
Classical domination in graphs is a well studied area within graph theory. For a graph $G,$ we consider $D\subset V(G)$ to be a \emph{dominating set} if every member of $V(G) \setminus D$ is adjacent to at least one member of $D.$ The \emph{domination number} of $G,$ denoted $\gamma(G),$ is the minimum cardinality of a dominating set. Define $w:V(G)\times V(G)\rightarrow \mathbb{R}$ to be a \emph{weight function} of $G.$ For $u,v \in V(G)$, we say that $u$ assigns weight $w(u,v)$ to $v$. Denote the weight assigned by a set of vertices $D$ to $v$ as $w(D,v)  := \sum_{d \in D} w(d,v),$ and similarly, the weight assigned by $d\in D$ to $H \subset V(G)$ as $w(d,H)  := \sum_{h \in H} w(d,h).$ The pair $(D,w)$ dominates $G$ if $w(D, v) \ge 1$ for every $v\in V(G).$ In the contex of classical domination, the pair $(D,w)$ dominates $G$ where $D$ is a dominating set and $w$ is the following function: \[w(u,v) = \begin{cases}
1 & \text{ if } u\in D \text{ and } uv \in E(G)\\
0 & \text{ otherwise.}
\end{cases}\]

A variant of domination, called \emph{exponential domination}, was first introduced in \cite{dankel}. Their motivation was to create a framework for a particular type of distance domination, one that would better model real world situations in which the influence of a selected vertex on other vertices decreases exponentially as their distance increases. There are two types of exponential domination;  \emph{non-porous} and \emph{porous}. In \emph{non-porous exponential domination,} exponential dominating vertices obstruct the influence of each other, whereas no there is no such obstruction in \emph{porous exponential domination.} More formally, the weight function for non-porous exponential domination is  \[w(u,v) = \left(\frac{1}{2} \right)^{\overline{\dist}(u,v) -1},\] where $\overline{\dist}(u,v)$ represents the length of the shortest $uv$ path that does not contain any internal vertices that are in the non-porous exponential dominating set. The \emph{non-porous exponential domination number} of $G$, denoted by $\gamma_e(G),$ is the cardinality of a minimum non-porous exponential dominating set. The weight function for porous exponential domination is  \[w^*(u,v) = \left(\frac{1}{2} \right)^{\dist(u,v) -1},\] where $\dist(u,v)$ represents the length of the shortest $uv$ path. The \emph{porous exponential domination number} of $G$, denoted by $\gamma^*_e(G),$ is the cardinality of a minimum porous exponential dominating set. 
 
Notice that exponential domination differs from the other variants of domination discussed in \cite{Haynes} due to the global influence exponential dominating vertices have on $V(G),$ whereas the dominating vertices of the variants of domination have a more local influence. The relatively few results \cite{anderson,Ayta,Bessy1,Bessy,dankel,Henning1,Henning} in this area has been attributed to the difficulty of working within the global nature of exponential domination. On relating exponential domination to classical domination, it is known that \cite{dankel}
\begin{equation}\label{PNPinequality}
 \exd(G) \le \gamma_e(G)  \le \gamma(G).
 \end{equation}
Let $[n] = \{ 1,2, \ldots, n\}.$ The \emph{consecutive circulant graph}, $C_{n , [\ell ]},$ is the set of $n$ vertices such that vertex $v$ is adjacent to vertex $v \pm i \mod n$ for each $i \in [\ell].$ Notice that $C_{n, [1]}$ is equivalent to $C_n$ and $C_{n,\left[\left\lfloor \frac{n}{2} \right\rfloor\right]}$ is equivalent to the complete graph $K_n.$ The following proposition gives an explicit formula for $\gamma_e(C_n).$ 

\begin{proposition}\cite{dankel} \label{dankelprop} { \rm For every integer $n\ge 3,$
\[ \gamma_e (C_n) = \begin{cases}
2 & \text{ if } n = 4 \\
\left\lceil \frac{n}{4} \right\rceil & \text{ if } n \neq 4.  
\end{cases} \]}
\end{proposition} 

No such formula has been determined for $\exd(C_n).$ In this paper, we show that that the porous and non-porous exponential domination number of $C_{n,[\ell]}$ are equivalent. Furthermore, in Theorem \ref{mainCirculantThm} when $\ell = 1$ and $m \ge 2,$ our results align with Proposition \ref{dankelprop} and fills the gap to $\exd(C_n).$ For the sake of simplicity, we will now refer to porous exponential domination as exponential domination, unless stated otherwise.

We still need a few more definitions and notation. Let $H$ be the Hamiltonian cycle of $C_{n,[\ell]},$ where the vertices $v, v+1 \mod n$ form an edge. Label the vertices of $C_{n,[\ell]}$ in the order of $H$ as $ V_H = \{0,1,\ldots, n-1\}.$ For $0\le i,j\le n-1,$ we denote $\dist_H(i,j)$ to be the length of the shortest path from $i$ to $j$ on $H$. See Figure \ref{C82} for an illustration of $C_{ 8 , [2] },$ 
\begin{figure}[htp!]
\begin{center}
\begin{tikzpicture}[scale = 1.8]
\draw  circle [radius=1cm];
\node[circle, fill = white, draw] (1) at ( 0:1cm ) {\tiny $2$};
\node[circle, fill = white, draw] (2) at ( 45:1cm ) {\tiny $1$};
\node[circle, fill = white, draw] (3) at ( 90:1cm) {\tiny $0$};
\node[circle, fill = white, draw] (4) at ( 135:1cm) {\tiny $7$};
\node[circle, fill = white, draw] (5) at ( 180:1cm) {\tiny $6$};
\node[circle, fill = white, draw] (6) at ( 225:1cm) {\tiny $5$};
\node[circle, fill = white, draw] (7) at ( 270:1cm) {\tiny $4$};
\node[circle, fill = white, draw] (8) at ( 315:1cm) {\tiny $3$};

\draw  (1)  to [bend left=45] (3) ;
\draw  (2)  to [bend left=45] (4) ;
\draw  (3)  to [bend left=45] (5) ;
\draw  (4)  to [bend left=45] (6) ;
\draw  (5)  to [bend left=45] (7) ;
\draw  (6)  to [bend left=45] (8) ;
\draw  (7)  to [bend left=45] (1) ;
\draw  (8)  to [bend left=45] (2) ;

\end{tikzpicture}
\caption{An illustration of $C_{8,[2]}$ }
\label{C82}
\end{center}
\end{figure}
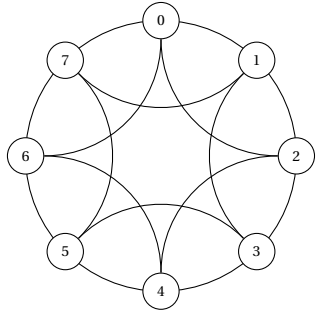 
with the defined labeling. With respect to $V(C_{n,[\ell]}),$ denote the interval $[i,j]$ as the set of increasing consecutive integers modulo $n$ from $i$ to $j.$ Let $ \I = \bigcup_{i=0}^{m-1} I_i$ be the consecutive partition around $H.$ For any exponential dominating set $D,$ let $f_k(D,\I) := | I_k \cap D |.$ Also define $z(D,\I) := \{ i : f_i(D,\I) = 0\},$ $Z(D,\I) := |z(D,\I) |,$ and $ \displaystyle f^*(D, \I ) := \max\limits_{0\le i \le m-1} f_i(D, \I).$   

Our main result is the following theorem, whose proof is shown in Section \ref{main} .
\begin{theorem}\label{mainCirculantThm}
{\rm Let $n = (3\ell + 1)m + r,$ for $0\le r \le 3\ell.$  Then \[ \exd(C_{n,[\ell]}) = \gamma_e(C_{n,[\ell]}) = \left\lceil \frac{n}{3\ell + 1} \right\rceil.\]}
\end{theorem}
We now give a brief outline of the proof for Theorem \ref{mainCirculantThm}. Through the use of the remarks and lemmas in Section \ref{minor}, we show that the above equality holds when $(3\ell+1)$ divides $n$. Additionally, the structure of the exponential domination set in this case is shown to be unique up to isomorphism. The main result is proven by exploiting the uniqueness of the exponential domination set when $(3\ell+1)$ divides $n,$ and applying (\ref{PNPinequality}).
\end{section}

%---------------------------------------------------------------------------------------------------------------------------------------------------------------------------------------------------------------------------------------

%%%%%%%%%%%%%%%%%%%%%%%%%%%%%%%%%%%%%%%%%%%%%%
%%%%%%%% CIRCULANT GRAPHS %%%%%%%%%%%%%%%%%%%%%%%%%%
%%%%%%%%%%%%%%%%%%%%%%%%%%%%%%%%%%%%%%%%%%%%%%

\begin{section}{Exponential domination of consecutive circulants} \label{Circulant}
In this section we prove Theorem \ref{mainCirculantThm}, which determines the explicit non-porous and porous exponential domination number of $C_{n,[\ell]},$ and shows that these numbers are equivalent. In the first subsection, we remark on minor results and provide lemmas used to prove the main results. The main results and their proofs are given in the second subsection. 

\begin{subsection}{Minor Results and Lemmas}\label{minor}

The following remarks and lemmas appear in the order they are referenced in the proofs of Theorems \ref{lmindom} and \ref{mainCirculantThm}.

\begin{remark}\label{computeDist}{\rm
Consider $u,v \in V(C_{n,[\ell]}).$ Throughout the paper, there will be a need to refer to $\dist(u,v).$ Notice that, \[ \dist(u,v) = \left\lceil \frac{\dist_H(u,v)}{\ell} \right\rceil.\]} 
\end{remark}

\begin{remark}\label{GenintDom}{\rm 
Suppose that $[a,b]$ is an interval on $H$ such that $a< b$ and $\dist_H(a,b) \le 3\ell +1.$ For notational simplicity, consider the interval $[0,3\ell +1].$ Then $0$ and $3\ell + 1$ dominates $[1, \ell ]$ and $[2\ell + 1, 3\ell ],$ respectively, and both $0$ and $3\ell + 1$ contribute weight $\tfrac{1}{2}$ to $[ \ell + 1, 2\ell ].$ Therefore $[0, 3\ell +1]$ is exponentially dominated by $0$ and $3\ell +1.$ This shows that $a,b$ exponentially dominates $[a,b].$}
\end{remark}

\begin{remark}\label{DminSet}{\rm
For $n = (3\ell +1)m,$ with $m\ge 2,$ let $D$ be a minimum exponential dominating set for $C_{n,[\ell]}.$ Fix vertex $i \in V_H$ and construct set $D^* = \{ i + (3\ell+1)t \mod n : 0 \le t \le m-1 \}.$ Through the application of Remark \ref{GenintDom}, $D^*$ forms an exponential dominating set for which $|D| \le |D^*| = m .$ 
}\end{remark}

\begin{remark}\label{f*Zeros}{ \rm
Let $D$ be a minimum exponential dominating set for $C_{n,[\ell]}.$ From Remark \ref{DminSet}, we have that $|D| \le m.$ Therefore, for every interval $I_k$ with $f_k(D)>1,$ there must exist $f_k(d) -1$ distinct intervals that contain no members of $D.$ This shows that $\sum_{k=0}^{m-1} f_k(D) \le m. $}
\end{remark}

In many of the cases for the proof of Theorem \ref{mainCirculantThm}, an exponential dominating set $D$ can be reduced to having at least one interval with no members of $D$ and the remaining intervals with exactly one member of $D.$ The following lemma gives results on $D$ in this situation. 
\begin{lemma}\label{newBC2}
{ \rm  Let $D \subset V(C_{n,[\ell]})$ and $\I$ be a partition such that $\I = \bigcup_{i=0}^{m-1} I_i, $ where $I_i = [ (3\ell +1)i, (3\ell+1) i + 3\ell]$ such that $f^*(D,\I) =1$ and $Z(D,\I) \ge 1.$ Let $d_i := I_i \cap D$ for $0\le i \le m-1,$ and consider $z \in z(D,\I).$ Then,
\begin{enumerate}[(i)]
\item\label{condition1}  $w^*(D , (3\ell +1)z +\ell) < \frac{6}{7}$ and  $w^*(D , (3\ell +1)z +2\ell) < \frac{6}{7},$
\item\label{condition2} $w^*(D \setminus d_k , (3\ell +1)z +\ell) < \frac{17}{28}$ and $w^*(D \setminus d_k , (3\ell +1)z +2\ell) < \frac{5}{14},$ for $k \equiv z +1 \mod m,$
\item\label{condition3} $w^*(D \setminus d_k , (3\ell +1)z +\ell) < \frac{5}{14}$ and $w^*(D \setminus d_k , (3\ell +1)z +2\ell) < \frac{17}{28},$ for $k \equiv z -1 \mod m.$ 
\item\label{condition4} $w^*(D \setminus d_k , (3\ell +1)z +\ell) < \frac{377}{448}$ and $w^*(D \setminus d_k , (3\ell +1)z +2\ell) < \frac{185}{224},$ for $k \equiv z +2 \mod m,$
\item\label{condition5} $w^*(D \setminus d_k , (3\ell +1)z +\ell) < \frac{185}{224}$ and $w^*(D \setminus d_k , (3\ell +1)z +2\ell) < \frac{377}{448},$ for $k \equiv z -2 \mod m.$ 
\end{enumerate} }
\end{lemma}

\begin{proof}
Let $\I$ be the partition such that $ \I = \bigcup_{i=0}^{m-1} I_i, $ where $I_i = [ (3\ell +1)i, (3\ell+1) i + 3\ell].$ For the sake of simplicity, let $f(D) = f(D,\I),$ $f^*(D) = f^*(D,\I),$ $z(D) = z(D,\I),$ and  $Z(D) = Z(D,\I).$ Without loss of generality, suppose that $0 \in z(D).$ Then the interval $I_0$ has $f_0(D) = 0.$ Among such $D,$ choose $D'$ to maximize $w^*(D', 2\ell)$ and let $k_0 = \left\lfloor \frac{m}{2} \right\rfloor.$ Then the choice of $D'$ implies that 
\[ I_k \cap D' = \begin{cases} (3\ell +1)k & \text{ for } k \le k_0 \\ (3\ell +1)k + 3\ell & \text{ for } k > k_0. \end{cases}\]
Consider $k \le k_0$ and notice that $\dist_H( 2\ell ,  (3\ell +1)k ) = (3\ell  + 1)k - 2\ell  =  (3k-2)\ell + k.$
By Remark \ref{computeDist},
\begin{equation}\label{DRdist} \dist( 2\ell ,  (3\ell +1)k ) = \left\lceil \frac{ (3k-2)\ell + k}{\ell} \right \rceil = 3k -2 + \left\lceil \frac{k}{\ell} \right\rceil. \end{equation}
Using the fact that $1 \le \left\lceil \frac{k}{\ell} \right\rceil,$ it follows that

\begin{eqnarray}\label{eq1}
\mathlarger\sum_{k = 1}^{k_0} w^*(I_k \cap (D' \setminus P_1'), 2\ell) &=& \mathlarger\sum_{k=1}^{k_0} \left( \frac{1}{2} \right)^{\dist( 2\ell ,  (3\ell +1)k ) -1}  
 < \quad  \mathlarger\sum_{k=1}^{\infty} \left( \dfrac{1}{2} \right)^{\dist( 2\ell ,  (3\ell +1)k ) -1}
 \quad = \quad \mathlarger\sum_{k=1}^{\infty} \left( \dfrac{1}{2} \right)^{3k -3 + \left\lceil \frac{k}{\ell} \right\rceil}   \notag \\ && \\
  &  \le & \mathlarger\sum_{k=1}^{\infty} \left( \dfrac{1}{2} \right)^{3k - 2} 
 \quad\quad\quad\quad=\quad \dfrac{1}{2} \mathlarger\sum_{t=0}^{\infty} \left( \dfrac{1}{2} \right)^{3t} 
  \quad \quad \quad\quad\quad \: \: \: \:= \quad \dfrac{4}{7}. \notag
\end{eqnarray}
Now consider $k> k_0$ and let $k' = m - k.$ Then with respect to $V_H,$
	\begin{eqnarray*} \dist_H( 2\ell ,  (3\ell +1)k+ 3\ell ) &=& \dist_H( (3\ell + 1)m + 2\ell ,  (3\ell +1)k + 3\ell ) \\
	&=& (3\ell + 1)m + 2\ell - (3\ell +1)k - 3\ell \\
	& = &(3k'-1)\ell + k'.
	\end{eqnarray*}	 
Again, applying Remark \ref{computeDist} gives,	
\begin{equation}\label{DLdist} \dist( 2\ell ,  (3\ell +1)k + 3\ell ) = \left\lceil \frac{ (3k'-1)\ell + k'}{\ell} \right \rceil = 3k' -1 + \left\lceil \frac{k'}{\ell} \right\rceil. \end{equation}
Notice since $k$ and $k'$ are counters, (\ref{DRdist}) and (\ref{DLdist}) only differ by $\ell.$ Furthermore, $k$ summing from $k = k_0 +1$ to $m-1$ is the same as summing $k'$ from $1$ to $m-k_0-1,$ which equals $k_0$ or $k_0 -1.$ Therefore we have shown that \[ w^*(D, 2\ell ) \quad \le \quad w^*(D',2\ell) \quad<\quad  \frac{3}{2} \sum_{k = 1}^{k_0} w^*(I_k \cap D', 2\ell)  \quad<\quad \frac{6}{7}.\]Applying a symmetric argument gives that $w^*(D, \ell ) < \frac{6}{7},$ and (\ref{condition1}) has been established. Let $d_k ' := I_k \cap D'$ and observe that $w^*( D \setminus d_k, \ell) \le w^*( D'\setminus d_k' , \ell)$ and $w^*( D \setminus d_k, 2\ell) \le w^*( D'\setminus d_k' , 2\ell).$ Notice that by construction, $d_1' = 3\ell +1$ and $d_2' = 6\ell +2.$ Then applying 
\[ \begin{array}{llllllllll} 	 
\dist_H(d_1', \ell )  &=&  3\ell + 1 - \ell & = & 2\ell + 1, \\
\dist_H(d_1' , 2\ell ) & = & 3\ell +1 - 2\ell &=& \ell + 1, \\
\dist_H(d_2', \ell )  &=&  6\ell + 2 - \ell & = & 5\ell + 2, \\
\dist_H(d_2' , 2\ell ) & = & 6\ell +2 - 2\ell &=& 4\ell + 2, \end{array} \]
with Remark \ref{computeDist}, and the fact that $1 = \left\lceil \frac{1}{\ell} \right\rceil$ and $1 \le \left\lceil \frac{2}{\ell} \right\rceil \le 2$  yields
\[ \begin{array}{lllllllllllll} 	
&& w^*(d_1', \ell ) &=& \left( \frac{1}{2} \right)^{\dist(3\ell +1, \ell) -1 } &=& \left( \frac{1}{2} \right)^{ 1+ \left\lceil \frac{1}{\ell} \right\rceil } &=& \frac{1}{4},\\
&& w^*(d_1', 2\ell ) &=& \left( \frac{1}{2} \right)^{\dist(3\ell +1, 2\ell) -1 } &=& \left( \frac{1}{2} \right)^{ \left\lceil \frac{1}{\ell} \right\rceil } &=& \frac{1}{2} \\
\frac{1}{64} & \le & w^*(d_2', \ell ) &=& \left( \frac{1}{2} \right)^{\dist(6\ell +2, \ell) -1 } &=& \left( \frac{1}{2} \right)^{ 4+ \left\lceil \frac{2}{\ell} \right\rceil } &\le& \frac{1}{32},\\
\frac{1}{32} & \le & w^*(d_2', 2\ell ) &=& \left( \frac{1}{2} \right)^{\dist(6\ell +2, 2\ell) -1 } &=& \left( \frac{1}{2} \right)^{ 3+\left\lceil \frac{2}{\ell} \right\rceil } &\le& \frac{1}{16}.\\
\end{array} \]
Therefore,
\[ \begin{array}{llllllllll} 
w^*( D \setminus d_1 , \ell ) & \le & w^*( D' \setminus d_1' , \ell ) &<& \frac{6}{7} &-& w^*(d_1' , \ell) & = & \frac{17}{28}, \\ & &&&&&&& \\
w^*( D \setminus d_1 , 2\ell ) & \le &  w^*( D' \setminus d_1' , 2\ell ) &<& \frac{6}{7} &-& w^*(d_1' , 2\ell) &=& \frac{5}{14}, \\ & &&&&&&& \\
w^*( D \setminus d_2 , \ell ) & \le & w^*( D' \setminus d_2' , \ell ) &<& \frac{6}{7} &-& w^*(d_2' , \ell) & = & \frac{377}{448}, \\ & &&&&&&& \\
w^*( D \setminus d_2 , 2\ell ) & \le &  w^*( D' \setminus d_2' , 2\ell ) &<& \frac{6}{7} &-& w^*(d_2' , 2\ell) &=& \frac{185}{224}. \\ & &&&&&&& \\
\end{array} \]
Therefore (\ref{condition2}) and (\ref{condition4}) have been verified. Observe that (\ref{condition3}) and (\ref{condition5}) are a symmetric arguments to (\ref{condition2}) and (\ref{condition4}), respectively.
\end{proof}

Given an exponential dominating set $D,$ the following algorithm details the process in how to construct a new exponential dominating set of the same size. With respect to $D,$ this new exponential dominating set has less intervals that contain no exponential dominating vertices, or  has less exponential dominating vertices contained in each interval.   
\begin{algorithm}\label{shiftalg} {\rm
Consider $D,$ an exponential dominating set for $C_{n,[\ell]},$ and the partition $\I$ such that $\I = \bigcup_{i=0}^{m-1} I_i, $ where $I_i = [ (3\ell +1)i, (3\ell+1) i + 3\ell].$ For the sake of simplicity, let $f(D) = f(D,\I),$ $f^*(D) = f^*(D,\I)$ and $Z(D) = Z(D,\I).$ Suppose that $3\le f^*(D).$ Observe that by Remark \ref{f*Zeros}, $Z(D) \ge 2.$ Without loss of generality, assume that for $0<b\le m,$ the intervals $I_0 $ and $I_b$ have that $f_0(D) = f_b(D) = 0.$ Find the interval $I_a $ such that $a = \min\{ 1,2,\ldots, b-1\}$ and $f_a(D) \ge 3.$ Furthermore, assume that the remaining $0<i<b$ have $f_i(D ) \ge 1.$ Without loss of generality suppose that $a \le b-a$ (use a reflection if necessary). Observe that there are at least $f^*(D) + b - 2$ exponential dominating vertices contained in $[0 , (3\ell +1)b + 3\ell].$ We identify the three closest members of $I_a \cap D$ to $(3\ell +1) a + 2\ell,$ and the closest member to $(3\ell +1)i + 2\ell$ within $I_i \cap D$ for $1\le i\neq a \le b,$ to be defined as $P = \{d_0, d_1, \ldots, d_b \}$ for which  
\[ d_i \in \begin{cases}
I_{i+1}\cap D & \text{ if } 0 \le i \le a-2 \\
I_a\cap D & \text{ if } a-1 \le i\le a + 1 \\
I_{i-1 }\cap D & \text{ if } a + 2 \le i \le  b.
\end{cases} \] 
Then define $S = \{s_0, s_1,\ldots, s_b\},$ such that $ s_t = (3\ell +1)t + 2\ell,$ and output the set $D' = (D \setminus P ) \cup S.$ }
 \end{algorithm}
 
 \begin{lemma}\label{shiftalgproof}{ \rm
 Given an exponential dominating set $D \subset V(C_{n,[\ell]} ),$ and the partition $\I$ so that $\I =  \bigcup_{i=0}^{m-1} I_i, $ where $I_i = [ (3\ell +1)i, (3\ell+1) i + 3\ell]$ and $3\le f^*(D,\I) \le 3\ell +1,$ Algorithm \ref{shiftalg} outputs the exponential dominating set $D'$ such that $|D| = |D'|,$ $Z(D',\I) = Z(D,\I) -2,$ and $f^*(D',\I) -2 \le f^*(D',\I) \le f^*(D,\I).$ 
 }
 \end{lemma}
 
 \begin{proof}
 For the sake of simplicity, let $f(D) = f(D,\I),$ $f^*(D) = f^*(D,\I)$ and $Z(D) = Z(D,\I).$  Notice that by construction, $|D| = |D'|$ and $\dist_H(s_i, s_{i+1} ) = 3\ell +1$ for each consecutive pair $s_i, s_{i+1} \in S.$ Through applications of Remark \ref{GenintDom}, all the vertices in $[s_0, s_b]$ are exponentially dominated by vertices of $S$. Let $V' = V(C_{n, [\ell] } ) \setminus [s_0, s_b].$ We define $V_L = \{ v \in V' : \dist_H( v, d_{a+1} ) \le \dist_H(v, d_{a+1} +1 ) \}$ and $V_R = V' \setminus V_L.$ There are four cases:
 
 \begin{enumerate}
 \item \label{vinV_L} Consider the case when $v \in V_L.$  Then $w^*(d_i, v) \ge w^*(d_{i+1} , v)$ for $0 \le i \le b-1.$  By construction, $ w^*(d_{a-2}, v) + w^*(d_{a-1} , v) \le 2w^*(d_{a-2},v) \le w^*(s_{a-2} , v)$ and $w^*(d_{a}, v) + w^*(d_{a+1} , v) \le 2w^*(d_{a},v) \le w^*(s_{a-1} , v).$ Additionally, $ w^*(d_i , v) \le w^*(s_i , v)$ for $0 \le i \le a-3$ and $w^*(d_i , v) \le w^*(s_{i-2} , v),$ for $a+2\le i \le b.$ Figure \ref{vinV_L1} visually shows that $w^*(P\cap D, v) <  w^*(S,v).$ Then putting it together gives that
\[  \begin{array}{lllllllllll}
 w^*(P\cap D, v) &=& \mathlarger\sum_{k=0}^{a-3} w^*(d_k,v) &+&  \mathlarger\sum_{k=a-2}^{a-1} w^*(d_k,v) &+& \mathlarger\sum_{k=a}^{a+1} w^*(d_k,v) &+& \mathlarger\sum_{k = a+2}^b w^*(d_k,v) \\	
 		& <&  \mathlarger\sum_{k=0}^{a-3} w^*(s_k,v) &+& w^*(s_{a-2},v) &+& w^*(s_{a-1},v) &+&\mathlarger \sum_{k = a}^{b-2} w^*(s_k,v) \\
		& <& w^*(S,v).
\end{array} \]
  
\begin{figure}[htp!]
\begin{center} 
\begin{tikzpicture}[scale=1]

\node (a) at (0,3) {$I_0$};
\node (b) at (1,3) {$I_1$};
\node(c) at (2,3) {$I_2$};
\node (d) at (3.5,3) {$I_3$};
\node (e) at (5,3) {$I_4$};
\node (f) at (6,3) {$I_5$};
\node (g) at (7,3) {$I_6$};
\node (h) at (8,3) {$I_7$};

\node (a1) at (0,2) {$\emptyset$};
\node (b1) at (1,2) {$d_0$};
\node(c1) at (2,2) {$d_1$};
\node (d1) at (3.5,2) {$d_2,d_3,d_4$};
\node (e1) at (5,2) {$d_5$};
\node (f1) at (6,2) {$d_6$};
\node (g1) at (7,2) {$d_7$};
\node (h1) at (8,2) {$\emptyset$};

\node (a2) at (0,1) {$s_0$};
\node (b2) at (1,1) {$s_1$};
\node(c2) at (2,1) {$s_2$};
\node (d2) at (3.5,1) {$s_3$};
\node (e2) at (5,1) {$s_4$};
\node (f2) at (6,1) {$s_5$};
\node (g2) at (7,1) {$s_6$};
\node (h2) at (8,1) {$s_7$};

\draw[ -> ] (b1) -- (a2);

\tikzset{
    position label/.style={
       below = 3pt,
       text height = 1.5ex,
       text depth = 1ex
    },
   brace/.style={
     decoration={brace, mirror},
     decorate
   }
}

\node (zz1) at (2.5, 1.5) {};
\node[position label] (z1) at (2 ,2.5) {};
\node[position label] (z2) at (3 ,2.5) {};
\draw [brace] (z1.south) -- (z2.south);

\draw[ -> ] (zz1) -- (b2);

\node (zz2) at (3.5, 1.5) {};
\node[position label] (z3) at (3.5 ,2.5) {};
\node[position label] (z4) at (4 ,2.5) {};
\draw [brace] (z3.south) -- (z4.south);

\draw[ -> ] (zz2) -- (c2);
\draw[ -> ] (e1) -- (d2);
\draw[ -> ] (f1) -- (e2);
\draw[ -> ] (g1) -- (f2);

\end{tikzpicture}
\caption{Illustration of Case \ref{vinV_L} with $a =3, b = 7$} 
\label{vinV_L1}
\end{center}
\end{figure}
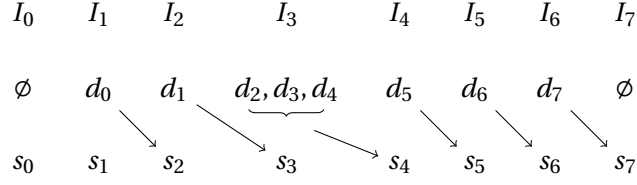 
	
\item \label{case1vinV_R} Consider the case when $v \in V_R$ such that $w^*(d_i, v) \le w^*(d_{i+1} , v)$ for $0 \le i \le b-1.$  Then $\dist_H(d_{a+1}, s_{a+1}) \ge 2\ell +1,$ which implies that $\sum_{k = a-1}^{a+1} w^*(d_k,v) \le 3 w^*(d_{a+1},v) < w^*(s_{a+1},v).$ Also by construction, $w^*(d_i , v) \le  w^*(s_{i+2} , v),$ for $ 0 \le i \le a-2$ and $w^*(d_i , v) \le w^*(s_{i} , v),$ for $a+2\le i \le b.$ Figure \ref{vinV_R2} visually shows that $w^*(P\cap D, v) < w^*(S,v).$ Then it follows that
\[ \begin{array}{llllllllllll}
 w^*(P\cap D, v) &=& \mathlarger\sum_{k=0}^{a-2} w^*(d_k,v) &+&  \mathlarger\sum_{k=a-1}^{a+1} w^*(d_k,v) &+& \mathlarger\sum_{k=a+2}^b w^*(d_k,v)  \\	
 		& <&  \mathlarger\sum_{k=2}^{a} w^*(s_k,v) &+& w^*(s_{a+1},v)  &+& \mathlarger\sum_{k = a+2}^b w^*(s_k,v) \\
		& <& w^*(S,v).
\end{array} \]

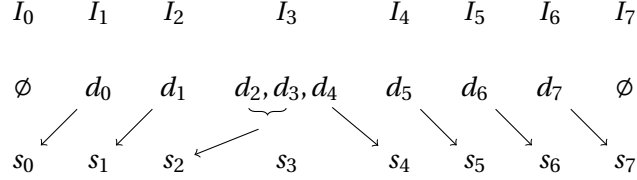
\begin{figure}[htp!]
\begin{center} 
\begin{tikzpicture}[scale=1]

\node (a) at (0,3) {$I_0$};
\node (b) at (1,3) {$I_1$};
\node(c) at (2,3) {$I_2$};
\node (d) at (3.5,3) {$I_3$};
\node (e) at (5,3) {$I_4$};
\node (f) at (6,3) {$I_5$};
\node (g) at (7,3) {$I_6$};
\node (h) at (8,3) {$I_7$};

\node (a1) at (0,2) {$\emptyset$};
\node (b1) at (1,2) {$d_0$};
\node(c1) at (2,2) {$d_1$};
\node (d1) at (3.5,2) {$d_2,d_3,d_4$};
\node (e1) at (5,2) {$d_5$};
\node (f1) at (6,2) {$d_6$};
\node (g1) at (7,2) {$d_7$};
\node (h1) at (8,2) {$\emptyset$};

\node (a2) at (0,1) {$s_0$};
\node (b2) at (1,1) {$s_1$};
\node(c2) at (2,1) {$s_2$};
\node (d2) at (3.5,1) {$s_3$};
\node (e2) at (5,1) {$s_4$};
\node (f2) at (6,1) {$s_5$};
\node (g2) at (7,1) {$s_6$};
\node (h2) at (8,1) {$s_7$};

\draw[ -> ] (b1) -- (c2);
\draw[ -> ] (c1) -- (d2);
\draw[ -> ] (e1) -- (f2);
\draw[ -> ] (f1) -- (g2);
\draw[ -> ] (g1) -- (h2);

\tikzset{
    position label/.style={
       below = 3pt,
       text height = 1.5ex,
       text depth = 1ex
    },
   brace/.style={
     decoration={brace, mirror},
     decorate
   }
}

\node (zz1) at (3.75, 1.5) {};
\node[position label] (z1) at (3 ,2.5) {};
\node[position label] (z2) at (4 ,2.5) {};
\draw [brace] (z1.south) -- (z2.south);

\draw[ -> ] (zz1) -- (e2);

\end{tikzpicture}
\caption{Illustration of Case \ref{case1vinV_R} with $a =3, b = 7.$} 
\label{vinV_R2}
\end{center}
\end{figure} 
\item \label{case2vinV_R} Consider the case when $v \in V_R$ such that $w^*(d_i, v) \ge w^*(d_{i+1} , v)$ for $0 \le i < a$ and $w^*(d_i, v) \le w^*(d_{i+1} , v)$ for $a+1 \le i < b-1.$ By construction, $w^*(d_{a-1} , v) + w^*(d_a , v) \le 2w^*(d_{a-1} , v) \le w^*(s_{a-1} , v).$ Additionally we have that $ w^*(d_i , v) \le w^*(s_i , v),$  for $ 0 \le i \neq a-1, a \le b.$ Figure \ref{vinV_R3a} visually shows that $  w^*(P\cap D, v) < w^*(S,v).$ Putting it together gives
\begin{eqnarray*}
 w^*(P\cap D, v) &=& \sum_{k=0}^{a-2} w^*(d_k,v)+  \sum_{k=a-1}^{a} w^*(d_k,v) + \sum_{k=a+1}^b w^*(d_k,v)  \\	
 		& <&  \sum_{k=0}^{a-2} w^*(s_k,v) + w^*(s_{a-1},v)  + \sum_{k = a+1}^b w^*(s_k,v) \\
		& <& w^*(S,v).
\end{eqnarray*}
 \begin{figure}[htp!]
\begin{center} 
\begin{tikzpicture}[scale=1]

\node (a) at (0,3) {$I_0$};
\node (b) at (1,3) {$I_1$};
\node(c) at (2,3) {$I_2$};
\node (d) at (3.5,3) {$I_3$};
\node (e) at (5,3) {$I_4$};
\node (f) at (6,3) {$I_5$};
\node (g) at (7,3) {$I_6$};
\node (h) at (8,3) {$I_7$};

\node (a1) at (0,2) {$\emptyset$};
\node (b1) at (1,2) {$d_0$};
\node(c1) at (2,2) {$d_1$};
\node (d1) at (3.5,2) {$d_2,d_3,d_4$};
\node (e1) at (5,2) {$d_5$};
\node (f1) at (6,2) {$d_6$};
\node (g1) at (7,2) {$d_7$};
\node (h1) at (8,2) {$\emptyset$};
\node(z) at (3.8 , 2){$\phantom{d_4}$};

\node (a2) at (0,1) {$s_0$};
\node (b2) at (1,1) {$s_1$};
\node(c2) at (2,1) {$s_2$};
\node (d2) at (3.5,1) {$s_3$};
\node (e2) at (5,1) {$s_4$};
\node (f2) at (6,1) {$s_5$};
\node (g2) at (7,1) {$s_6$};
\node (h2) at (8,1) {$s_7$};

\draw[ -> ] (b1) -- (a2);
\draw[ -> ] (c1) -- (b2);
\draw[ -> ] (e1) -- (f2);
\draw[ -> ] (f1) -- (g2);
\draw[ -> ] (g1) -- (h2);

\tikzset{
    position label/.style={
       below = 3pt,
       text height = 1.5ex,
       text depth = 1ex
    },
   brace/.style={
     decoration={brace, mirror},
     decorate
   }
}
\node (zz1) at (3.25, 1.5) {};
\node[position label] (z1) at (3 ,2.5) {};
\node[position label] (z2) at (3.5 ,2.5) {};
\draw [brace] (z1.south) -- (z2.south);

\draw[ -> ] (zz1) -- (c2);
\draw[ -> ] (z) -- (e2);
\end{tikzpicture}
\caption{Illustration of Case \ref{case2vinV_R} with $a =3, b = 7.$} 
\label{vinV_R3a}
\end{center}
\end{figure}
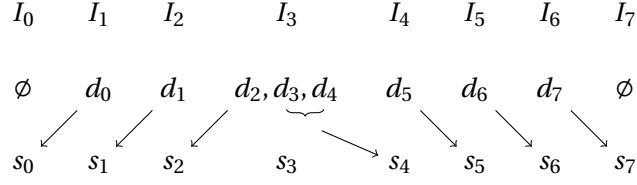 

\item \label{case3vinV_R} Consider the case when $v \in V_R$ such that  $w^*(d_i, v) \ge w^*(d_{i+1} , v)$ for $0 \le i < a-1$ and $w^*(d_i, v) \le w^*(d_{i+1} , v)$ for $a \le i < b-1.$ By construction, $w^*(d_a , v) + w^*(d_{a+1} , v) \le 2w^*(d_{a+1} , v) \le w^*(s_{a+1} , v). $ Additionally we have that $ w^*(d_i , v) \le w^*(s_i , v),$  for $ 0 \le i \neq a, a+1 \le b.$ Figure \ref{vinV_R3} visually shows that $w^*(P\cap D, v) <  w^*(S,v).$ Putting it together gives
\[\begin{array}{lllllllll}
 w^*(P\cap D, v) &=& \mathlarger\sum_{k=0}^{a-1} w^*(d_k,v) &+&  \mathlarger\sum_{k=a}^{a+1} w^*(d_k,v) &+& \mathlarger\sum_{k=a+2}^b w^*(d_k,v)  \\	
 		& <&  \sum_{k=0}^{a-1} w^*(s_k,v) &+& w^*(s_{a+1},v)  &+& \mathlarger\sum_{k = a+2}^b w^*(s_k,v) \\
		& <& w^*(S,v).
\end{array} \]

 \begin{figure}[htp!]
\begin{center} 
\begin{tikzpicture}[scale=1]

\node (a) at (0,3) {$I_0$};
\node (b) at (1,3) {$I_1$};
\node(c) at (2,3) {$I_2$};
\node (d) at (3.5,3) {$I_3$};
\node (e) at (5,3) {$I_4$};
\node (f) at (6,3) {$I_5$};
\node (g) at (7,3) {$I_6$};
\node (h) at (8,3) {$I_7$};

\node (a1) at (0,2) {$\emptyset$};
\node (b1) at (1,2) {$d_0$};
\node(c1) at (2,2) {$d_1$};
\node (d1) at (3.5,2) {$d_2,d_3,d_4$};
\node (e1) at (5,2) {$d_5$};
\node (f1) at (6,2) {$d_6$};
\node (g1) at (7,2) {$d_7$};
\node (h1) at (8,2) {$\emptyset$};
\node(z) at (3 , 2){$\phantom{d_2}$};

\node (a2) at (0,1) {$s_0$};
\node (b2) at (1,1) {$s_1$};
\node(c2) at (2,1) {$s_2$};
\node (d2) at (3.5,1) {$s_3$};
\node (e2) at (5,1) {$s_4$};
\node (f2) at (6,1) {$s_5$};
\node (g2) at (7,1) {$s_6$};
\node (h2) at (8,1) {$s_7$};

\draw[ -> ] (b1) -- (a2);
\draw[ -> ] (c1) -- (b2);
\draw[ -> ] (e1) -- (f2);
\draw[ -> ] (f1) -- (g2);
\draw[ -> ] (g1) -- (h2);

\tikzset{
    position label/.style={
       below = 3pt,
       text height = 1.5ex,
       text depth = 1ex
    },
   brace/.style={
     decoration={brace, mirror},
     decorate
   }
}
\node (zz1) at (3.85, 1.5) {};
\node[position label] (z1) at (3.5,2.5) {};
\node[position label] (z2) at (4 ,2.5) {};
\draw [brace] (z1.south) -- (z2.south);
\draw[ -> ] (zz1) -- (e2);
\draw[ -> ] (z) -- (c2);
\end{tikzpicture}
\caption{Illustration of Case \ref{case3vinV_R} with $a =3, b = 7.$} 
\label{vinV_R3}
\end{center}
\end{figure}
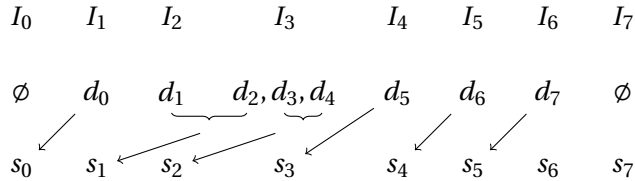 
 \end{enumerate}
 In each instance, we have shown that $|D| = |D'|,$ $w^*(D', i ) \ge 1$ for every $i \in [s_0, s_b],$ and  $w^*(D, v ) \le w^*(D',v)$ for every $v \in V'.$ Therefore $D'$ is an exponential dominating set. By construction, it follows that $I_0$ and $I_b$ have that $f_0(D') = f_0(D) +1$ and $f_b(D' ) = f_b(D) +1.$ Furthermore, all remaining $I_i$ where  $i \neq 0,b$ have that $f_i(D) = f_i(D').$ Therefore $Z(D') = Z(D) -2$ and $f^*(D') -2 \le f^*(D') \le f^*(D).$ 
\end{proof}

The following lemma shows that if $D\subset V(C_{n,[\ell]})$ has the property that one interval contains three members of $D,$ two intervals that contain no members of $D$, and all remaining intervals have one member of $D,$ then $D$ cannot be an exponential dominating set.

 \begin{lemma}\label{f*=3notDom}
{ \rm Consider $D\subset V(C_{n,[\ell]})$ and the partition $\I$ such that $ \I =  \bigcup_{i=0}^{m-1} I_i, $ where $I_i = [ (3\ell +1)i, (3\ell+1) i + 3\ell]$ and $f^*(D,\I) = 3.$ Assume that  $I_i ,I_j \subset \I$ are cyclically consecutive intervals for which $f_i(D,\ell) = f_j(D,\ell) = 0.$ Furthermore suppose that there exist $I_k \subset \I$ for which $i<k<j$ and $f_k(D,\ell) = 3,$ and all remaining intervals $I_t\subset \I$ have $f_t(D,\ell) = 1.$ Then $D$ cannot be an exponential dominating set. }
\end{lemma}
\begin{proof} For the sake of simplicity, let $f^*(D) = f^*(D,\I)$ and $Z(D) = Z(D,\I).$ Without loss of generality, assume that for $0<a<b < m,$ the intervals $I_0$ $I_a,$ and $I_b$ have $f_0(D) = f_b(D) = 0,$ $f_a(D) = 3.$ Furthermore, assume that the remaining intervals $I_t$ have that $f_t(D) = 1.$ Let $I_a \cap D = \{ d_1 , d_2, d_3\}$ and without loss of generality, consider $d_1, d_2.$ Notice that there is exactly one member of $D \setminus \{d_1, d_2\}$ in every nonempty interval, so $f^*(D \setminus \{d_1, d_2\}) = 1$ and $Z(D \setminus \{d_1, d_2\}) \ge 1.$ By (\ref{condition1}) of Lemma \ref{newBC2}, it follows that $w^*(D \setminus \{d_1, d_2\} ,  \ell ), w^*(D \setminus\{d_1, d_2\} , (3\ell +1)b + 2\ell ) < \frac{6}{7}.$ This implies that $w^*(\{d_1, d_2\} ,  \ell ),$ $w^*(\{d_1, d_2\} , (3\ell +1)b + 2\ell ) > \frac{1}{7}.$ The only case in which both conditions hold is when $a = 1$ and $b = 2.$ Among such $D,$ we choose $D'$ to maximize $w^*(D',\ell) + w^*(D', (3\ell+1)b + 2\ell).$ 
 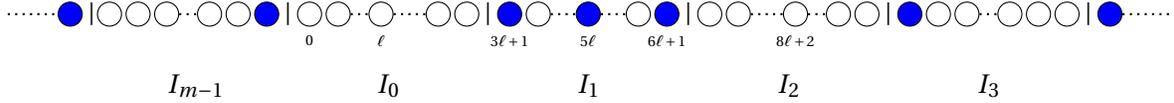
\begin{figure}[!htp]
\centering
\begin{tikzpicture}[scale = .95]
\node (z) at (-0.25, 0) {$|$};
\node (aa1) at (1.2,-1) {$I_{m-1}$};
\node[circle, fill = white, draw] (a) at (0,0) {};
\node[circle, fill = white, draw] (b) at (0.40,0) {};
\node[circle, fill = white, draw] (c) at (.8,0) {};
\node[circle, fill = white, draw] (e) at (1.4,0) {};
\node[circle, fill = white, draw] (f) at (1.8,0) {};
\node[circle, fill = blue, draw] (g) at (2.2,0) {};
\node (z1) at (2.5, 0) {$|$};
\draw[ thick, dotted ] (c) -- (e);

\node (aa1) at (3.9,-1) {$I_0$};
\node[circle, fill = white, draw, label = below:\tiny$0$] (a1) at (2.8,0) {};
\node[circle, fill = white, draw] (b1) at (3.2,0) {};
\node[circle, fill = white, draw, label = below:\tiny$\ell$] (c1) at (3.8,0) {};
\node[circle, fill = white, draw] (e1) at (4.6,0) {};
\node[circle, fill = white, draw] (g1) at (5,0) {};
\node (z3) at (5.3, 0) {$|$};
\draw[ thick, dotted ] (c1) -- (e1);

\draw[ thick, dotted ] (b1) -- (c1);
\node (aa1) at (6.7,-1) {$I_1$};
\node[circle, fill = blue, draw, label = below:\tiny$3\ell +1$] (a2) at (5.6,0) {};
\node[circle, fill = white, draw] (b2) at (6,0) {};
\node[circle, fill = blue, draw, label = below:\tiny$5\ell$] (c2) at (6.7,0) {};
\node[circle, fill = white, draw] (f2) at (7.4,0) {};
\node[circle, fill = blue, draw,  label = below:\tiny$6\ell +1$] (g2) at (7.8,0) {};
\node (z5) at (8.1, 0) {$|$};
\draw[ thick, dotted ] (b2) -- (c2) -- (f2);

\node (aa1) at (9.5,-1) {$I_2$};
\node[circle, fill = white, draw] (a3) at (8.4,0) {};
\node[circle, fill = white, draw] (b3) at (8.8,0) {};
\node[circle, fill = white, draw, label = below:\tiny$8\ell +2$] (c3) at (9.6,0) {};
\node[circle, fill = white, draw] (f3) at (10.2,0) {};
\node[circle, fill = white, draw] (g3) at (10.6,0) {};
\node (z7) at (10.9, 0) {$|$};
\draw[ thick, dotted ] (b3) -- (c3) -- (f3);

\node (aa1) at (12.3,-1) {$I_3$};
\node[circle, fill = blue, draw] (a4) at (11.2,0) {};
\node[circle, fill = white, draw] (b4) at (11.6,0) {};
\node[circle, fill = white, draw] (c4) at (12,0) {};
\node[circle, fill = white, draw] (d4) at (12.6,0) {};
\node[circle, fill = white, draw] (f4) at (13,0) {};
\node[circle, fill = white, draw] (g4) at (13.4,0) {};
\node (z8) at (13.7, 0) {$|$};
\draw[ thick, dotted ] (c4) -- (d4);

\node[circle, fill = blue, draw] (end1) at (14,0) {};

\node (start) at (-1.55,0){};
\node[circle, fill = blue, draw] (start1) at (-.55,0) {};

\node (end) at (15,0){};

\draw[ thick, dotted ] (start) -- (start1);
\draw[ thick, dotted ] (end1) -- (end);
\end{tikzpicture}
 \caption{Visualization of $D'$, with edges removed and members of $D'$  colored}
 \label{visualf*D=3,Z=2}
  \end{figure} 
See Figure \ref{visualf*D=3,Z=2} for an illustration of $D'.$ Let $k_0 =  \left\lfloor  \frac{m}{2} \right\rfloor +1$ and consider the case when $m$ is odd. Then, \[ I_k \cap D' = \begin{cases} \{ 3\ell + 1, 5\ell, 6\ell +1 \} & \text{ if } k = 1 \\ (3\ell + 1)k & \text{ if } 2< k \le k_0 \\ (3\ell +1)k + 3\ell  & \text{ if }  k_0< k \le m. \end{cases} \]
If $m$ is even, \[ I_k \cap D' = \begin{cases} \{ 3\ell + 1, 5\ell, 6\ell +1 \} & \text{ if } k = 1 \\ (3\ell + 1)k & \text{ if } 2< k < k_0 \\ (3\ell +1)k + 2\ell-1 & \text{ if } k = k_0 \\ (3\ell +1)k + 3\ell  & \text{ if } k_0 < k \le m. \end{cases} \]
We now compute the length of the shortest path from $\ell$ to $5\ell, 6\ell +1$ and from $8\ell +2$ to $ 3\ell +1, 5\ell.$  Then notice that with respect to $V_H,$ $\dist_H( 5\ell , \ell ) = 4\ell,$ $ \dist_H( 6\ell +1, \ell ) = 5\ell +1,$ $\dist_H( 3\ell +1 , 8\ell +2 ) = 5\ell +1,$ and $\dist_H(  5\ell  , 8\ell +2  )  = 3\ell +2.$ Using Remark \ref{computeDist} and the fact that $1 = \left\lceil \frac{1}{\ell} \right\rceil$ and $1\le  \left\lceil \frac{2}{\ell} \right\rceil \le 2,$ it follows that
\[ \begin{array}{llllllllllll}
 w^*( 5\ell, \ell ) &=& \left( \frac{1}{2} \right)^{\dist( 5\ell ,  \ell ) -1} &=& \left( \frac{1}{2} \right)^3 &=& \frac{1}{8} \\
 w^*( 6\ell +1, \ell ) &=& \left( \frac{1}{2} \right)^{\dist( 6\ell +1 ,  \ell ) -1} &= &\left( \frac{1}{2} \right)^{4+ \left\lceil \frac{1}{\ell} \right\rceil } &=&  \frac{1}{32}  \\
 w^*( 3\ell +1, 8\ell +2) &=& \left( \frac{1}{2} \right)^{\dist( 3\ell +1, 8\ell +2) -1} &=& \left( \frac{1}{2} \right)^{4+ \left\lceil \frac{1}{\ell} \right\rceil } &=&  \frac{1}{32}  \\
 w^*( 5\ell , 8\ell +2 ) &=& \left( \frac{1}{2} \right)^{\dist( 5\ell , 8\ell +2) -1} &=& \left( \frac{1}{2} \right)^{2+ \left\lceil \frac{2}{\ell} \right\rceil } &\le&  \frac{1}{8},
 \end{array} \]
 Therefore $ w^*( \{5\ell, 6\ell+ 1 \} , \ell ), w^*( \{3\ell +1, 5\ell \} , 8\ell +2 ) \le \frac{5}{32}.$ Regardless of the parity of $m,$ there is exactly one member of $D' \setminus\{3\ell +1, 5\ell \}$ and one member of $D' \setminus\{5\ell, 6\ell+ 1\}$ in every nonempty interval. Therefore $f^*(D' \setminus\{3\ell +1, 5\ell \}) = f^*(D' \setminus\{5\ell, 6\ell+ 1\}) = 1$ and $Z(D' \setminus\{3\ell +1, 5\ell \}),$ $Z(D' \setminus\{5\ell, 6\ell+ 1\}) \ge 0.$ Then by (\ref{condition4}) and (\ref{condition5}) of Lemma \ref{newBC2}, it follows that $ w^*(D' \setminus \{5\ell, 6\ell+ 1 \} , \ell ) ,w^*(D' \setminus \{3\ell +1, 5\ell \} , 8\ell +2 ) <  \frac{377}{448}.$ Therefore
 \[ \begin{array}{llllllll}
  w^*(D , \ell) &\le& w^*(D' \setminus \{5\ell, 6\ell+ 1 \}, \ell) &+& w^*( \{5\ell, 6\ell+ 1 \}, \ell) &=& \frac{447}{448}, \\ &&&&&& \\
  w^*(D, 8\ell +2 ) &\le&  w^*(D' \setminus \{3\ell +1, 5\ell \} , 8\ell +2 ) &+& w^*( \{3\ell +1, 5\ell \} , 8\ell +2 ) &=& \frac{447}{448}.
  \end{array}\]
Putting it together gives that $w^*(D , \ell) +   w^*(D, 8\ell +2 ) =\frac{447}{224} < 2.$ As it is not possible for $\ell$ and $8\ell +2$ to receive sufficient weight from $D',$ it follows that $D$ cannot be an exponential dominating set.
\end{proof}

Consider the situation when there are at most two members of $D \subset V(C_{n,[\ell]})$ in each interval. The following lemma shows that for every interval that contains no members of $D,$ there must be an adjacent interval that contains two members of $D.$

\begin{lemma}\label{f_k=Z}
{ \rm Let $D$ be a minimum exponential dominating set for $C_{n,[\ell]}$ and $\I$ be a partition such that $\I =  \bigcup_{i=0}^{m-1} I_i, $ where $I_i = [ (3\ell +1)i, (3\ell+1) i + 3\ell]$ and $f^*(D,\I) = 2.$ Then for every $0 \le j <m$ for which $f_j(D,\I) = 0,$ there exist $k \equiv j \pm 1 \mod m$ such that $f_k(D,\I) = 2.$ Moreover,  $| \{ k : f_k(D,\I) = 2 \} | = Z(D,\I).$  }
\end{lemma}

\begin{proof}
For the sake of simplicity, let $f(D) = f(D,\I),$ $f^*(D) = f^*(D,\I),$ $z(D) = z(D,\I),$ and  $Z(D) = Z(D,\I).$ Let $K = \{k : f_k(D) = 2\}.$ As Remark \ref{DminSet} shows that $|D| \le m,$ it follows that for every $k \in K$ there must exist a distinct $z \in z(D).$ Therefore we have that $|K| \le Z(D).$ Let $I_k \cap D=\{ d_k , d_k' \}$ for every $k \in K$ and let  $P = P_1 \cup P_2$ such that $P_1 = \{ d'_k : k \in K\}$ and $P_2 = \{ d_k : k \in K\}.$ Without loss of generality suppose that $0 \in z(D).$ Then the interval $I_0$ has that $f_0(D) =0.$ Let $\hat{D}$ be an exponential dominating set such that $f_0(\hat{D}) = 0,$ $f_1(\hat{D}) = f_{m-1}(\hat{D}) = 1,$ and $f_i(\hat{D}) = 2$ for $2\le i \le m-2.$ Among such $\hat{D},$ choose $D'$ to maximize $w^*(D' , 2\ell).$ Let $I_k \cap D'=\{ s_k , s_k' \}$ such that $w^*(s_k , 2\ell) \le w^*(s_k', 2\ell)$ and without loss of generality, let $P_1' = \{ s'_k : 2\le k \le m-2\}.$ By construction, it follows that $ w^*(D\setminus P_1,2\ell) \le w^*(D' \setminus P_1' , 2\ell).$ Notice that there is exactly one member of $D' \setminus P_1'$ in every nonempty interval, so $f^*(D' \setminus P_1') = 1$ and $Z(D' \setminus P_1') \ge 1.$ Therefore by Lemma (\ref{condition1}) of \ref{newBC2}, $w^*( D' \setminus P_1' ,2\ell) <\frac{6}{7}.$ Putting it together gives $w^*( D \setminus P_1,2\ell) <\frac{6}{7}.$ Let  $k_0 = \left\lfloor \frac{m}{2} \right\rfloor,$ then the choice of $D'$ implies that   
 \[I_k \cap P_1'= \begin{cases} (3\ell +1)k & \text{ if  } 2 \le k \le k_0 \\
  (3\ell +1)k + 3\ell  & \text{ if  }  k_0 < k \le m-2.
  \end{cases} \]
Consider $2\le k \le k_0.$ Then using (\ref{DRdist}) and  (\ref{eq1}), it follows that 
\begin{eqnarray}\label{eq2}
 \sum_{k = 2}^{k_0} w^*(I_k \cap P_1', 2\ell)  <\quad   \sum_{k=1}^{\infty} \left( \dfrac{1}{2} \right)^{\dist( 2\ell ,  (3\ell +1)k ) -1} - \quad  \left( \dfrac{1}{2} \right)^{\dist( 2\ell ,  (3\ell +1)) -1}  \le  \quad\dfrac{4}{7} \quad-\quad \dfrac{1}{2} &=& \dfrac{1}{14}. 
 \end{eqnarray}
Now consider $k_0 < k \le m-2$ and let $k' = m - k.$ Notice since $k$ and $k'$ are counters, (\ref{DRdist}) and (\ref{DLdist}) only differ by $\ell.$ Furthermore, $k$ summing from $k = k_0 +1$ to $m-2$ is the same as summing $k'$ from $2$ to $m-k_0-1,$ which equals $k_0$ or $k_0 -1.$ Therefore using (\ref{eq2}), it follows that 
\begin{equation}
w^*(P_1, 2\ell ) \quad  \le\quad   w^*(P_1',2\ell) \quad \quad<\quad\quad  \frac{3}{2} \sum_{k = 2}^{k_0} w^*(I_k \cap P_1', 2\ell)  \quad <\quad  \frac{3}{28}. 
\end{equation}
Thus, $w^*(D, 2\ell ) = w^*(D\setminus P_1, 2\ell ) + w^*(P_1, 2\ell )< \frac{27}{28},$ which contradicts the assumption that $D$ is an exponential dominating set. Through a symmetric argument, it can be shown that $w^*(D, \ell ) < \frac{27}{28}.$ Therefore either $1 \in K,$ or $m-1 \in K.$ In general, this shows that for every $z \in z(D),$ there exist $k \in K$ such that $k \equiv z \pm 1 \mod m.$ Without loss of generality, suppose that $k \equiv z+1 \mod n.$ Then, $z-1 \mod n \not\in K,$ else there will exist $z_0 \in z(D)$ for which $w^*(D, (3\ell +1)z_0 + \ell) , w^*(D, (3\ell +1)z_0 + 2\ell) < \frac{27}{28}.$ Thus, $|K| = Z(D).$ 
\end{proof}

The next lemma extends the result of Lemma \ref{locations} by determining the location of exponential dominating vertices in the intervals to either side of an interval that contains no members of $D.$ 
 
 \begin{lemma}\label{locations}
{ \rm Let $D$ be an exponential dominating set for $C_{n,[\ell]}$ and $\I$ be a partition such that $\I = \bigcup_{i=0}^{m-1} I_i, $ where $I_i = [ (3\ell +1)i, (3\ell+1) i + 3\ell]$ and $f^*(D,\I) = 2.$ Consider the interval $I_j$ with $f_j(D,\I) = 0.$ Let $k\equiv j - 1 \mod m$ and $k' \equiv j +1 \mod m.$  Then there are either two members of $D$ contained in $ [ (3\ell+1)k +2\ell +1 , (3\ell+1)k+3\ell ] $ and one member of $D$ contained in $ [ (3\ell+1)k'  , (3\ell+1)k' +\ell -1] |,$ or two members of $D$ contained in $ [ (3\ell+1)k'  , (3\ell+1)k' +\ell -1 ] $ and one member of $D$ contained in $ [ (3\ell+1)k +2\ell +1 , (3\ell+1)k+3\ell ].$ }
\end{lemma}

\begin{proof}
For the sake of simplicity, let $f(D) = f(D,\I),$ $f^*(D) = f^*(D,\I),$ $z(D) =z(D,\I),$ and $Z(D) = Z(D,\I).$ Let $K = \{k : f_k(D) = 2\}$ and let $I_k \cap D=\{ d_k , d_k' \}$ for every $k \in K.$ Define $P = P_1 \cup P_2$ such that $P_1 = \{ d'_k : k \in K\}$ and $P_2 = \{ d_k: k \in K\}.$ Without loss of generality, consider $P_1$ and notice that there is exactly one member of $D\setminus P_1$ in every nonempty interval, so $f^*(D\setminus P_1 ) = 1$ and $Z(D \setminus P_1) \ge 1$. By (\ref{condition1}) of Lemma \ref{newBC2}, every $z\in z(D)$ has that $w^*(D \setminus P_1 , (3\ell +1)z + \ell ), w^*(D \setminus P_1 , (3\ell +1)z + 2\ell ) < \frac{6}{7}.$ To maintain that $(D,w^*)$ dominates $C_{n,[\ell]},$  it follows that $w^*(P_1 , (3\ell +1)z + \ell ), w^*( P_1 , (3\ell +1)z + 2\ell ) > \frac{1}{7}.$ Without loss of generality assume $0 \in z(D).$ Then Lemma \ref{f_k=Z} shows that either $1 \in K$ or $m-1\in K$. Suppose $1 \in K$ and consider $d_1 \in P_2.$ Then by (\ref{condition2}) of Lemma \ref{newBC2}, $w^*( D \setminus( P_1 \cup d_1) , \ell ) < \frac{17}{28}$ and $w^*( D \setminus( P_1 \cup d_1) , 2\ell ) < \frac{5}{14}.$ To ensure that $\ell$ and $2\ell$ receive sufficient weight from $D,$ the following conditions must hold
\begin{eqnarray}
w^*( P_1 \cup d_1 , \ell)  &>& \frac{11}{28}, \label{ellclosevertex} \\
w^*( P_1 \cup d_1 , 2\ell) & >& \frac{9}{14}. \label{2ellclosevertex}
\end{eqnarray}
Since $w^*(P_1 ,  \ell ), w^*( P_1 , 2\ell ) > \frac{1}{7},$ it follows that $w^*( d_1,\ell ) \ge \frac{1}{4}$ and $w^*( d_1, 2\ell ) \ge \frac{1}{2},$ satisfy (\ref{ellclosevertex}) and (\ref{2ellclosevertex}). This implies that $d_1\in [ 3\ell+1,4\ell ].$ Let $d_{m-1} = I_{m-1} \cap D$ and note that a similar argument gives that $d_{m-1} \in [ (3\ell+1)(m-1) + 2\ell +1 , (3\ell+1)(m-1) + 3\ell].$ Additionally, through a similar argument with respect to $P_2,$ it can be shown that $d_1' \in [3\ell+1 , 4\ell ].$ Now consider the case when $m-1 \in K.$ For $d_{m-1}, d_{m-1}' \in I_{m-1} \cap D$ and $d_1 \in I_1 \cap D,$ a symmetric argument gives that $d_{m-1} , d_{m-1}' \in [ (3\ell+1)(m-1) + 2\ell +1 , (3\ell+1)(m-1) + 3\ell]$ and $d_1\in [ 3\ell+1,4\ell ].$
 \end{proof}

The following lemma shows that if there are two exponential dominating vertices that are within a certain distance of each other, then there exists a shift of these two vertices that creates a new exponential dominating set.
 
\begin{lemma}\label{2shift}{\rm
Let $D$ be an exponential dominating set for $C_{n , [\ell]}.$ Suppose that there exists $i, j \in D$ such that $i < j$ and $\dist_H(i, j) \le \ell +1.$ Consider $S = (V(C_{n, [\ell]}) \setminus D) \cap [ j+\ell ,  i -\ell ].$ Let $a_0,b_0 \in S$ so that $\dist_H(a_0, i-\ell) < \dist_H(a , i-\ell)$ for every $a \in S\setminus a_0$ and $\dist_H(b_0, j + \ell) < \dist_H(b , j + \ell)$  and for every $b \in S \setminus b_0.$ Then $D' = (D \setminus \{i,j\}) \cup \{a,b \}$ is an exponential dominating set.}
\end{lemma}
\begin{proof}
Consider $D' = (D \setminus \{i,j\}) \cup \{a_0,b_0 \}.$ As $\dist_H( i -\ell , j + \ell ) \le 3\ell +1,$  Remark \ref{GenintDom} shows that $ i-\ell , j+\ell $ exponentially dominates $[ i-\ell ,   j+\ell ].$ Then $w^*(D', u) \ge 1$ for all $u\in [a_0,b_0].$ Let $v \in V(C_{n,[\ell]}) \setminus [a_0,b_0] $ and without loss of generality, suppose that $w^*( i, v) \ge w^*( j,v).$ Observe that $w^*( i,v) + w^*( j,v) \le  2w^*( i,v) \le w^*( a_0  , v),$ which implies that $w^*(D , v) \le w^*(D',v).$ Thus $D'$ is an exponential dominating set.
\end{proof}
 \end{subsection}
 
 \begin{subsection}{Main Results}\label{main}
The main results of this paper consists of the following two theorems. Theorem \ref{lmindom} determines the structure of the minimum porous exponential dominating set for $ C_{n,[\ell]},$ when $3\ell+1$ divides $n.$ In this proof, all but one case is shown to either have a porous exponential dominating set that is not minimum, or to have a set of vertices that is not a porous exponential dominating set. Theorem \ref{mainCirculantThm} determines the explicit formula for $\gamma_e^*(C_{n,[\ell]})$ and $\gamma_e(C_{n,[\ell]}).$ In this proof Theorem \ref{lmindom} and Remark \ref{GenintDom} to determine a lower bound for $\exd(C_{n,[\ell]})$ and upper bound for $\gamma_e(C_{n,[\ell]}),$ respectively. Additionally (\ref{PNPinequality}) is used to link $ \gamma_e^*(C_{n,[\ell]})$ and $\gamma_e(C_{n,[\ell]}).$

\begin{theorem}\label{lmindom}{\rm
Let $n = (3\ell+1)m \ge 6\ell+2.$ Let $D$ is a minimum exponential dominating set for $C_{n,[\ell]},$ and $\I$ be a partition such that $ \I =  \bigcup_{i=0}^{m-1} I_i, $ where $I_i = [ (3\ell +1)i, (3\ell+1) i + 3\ell].$ Then $f^*(D,\ell) =1$ and $Z(D,\ell) = 0$ for any partition $\I.$ Furthermore, $D$ is unique up to isomorphism. }\end{theorem}

\begin{proof}
Let $D$ be an exponential dominating set for $C_{n,[\ell]}.$ For the sake of simplicity, let $f(D) = f(D,\I),$ $f^*(D) = f^*(D,\I)$ and $Z(D) = Z(D,\I).$ Through induction, we show the contrapositive of the statement: if $ 2 \le f^*(D)+ Z(D)  \le 3\ell +1,$ then $D$ cannot be a minimum exponential dominating set. 

\begin{enumerate}[BC 1]
 \item Suppose that $f^*(D) \ge 2$ and $Z(D)=0.$ Through counting the exponential dominating vertices, it follows that $|D| \ge  m + 1.$ Remark \ref{DminSet}, shows that there exists an exponential dominating set $D^*$ for $C_{n,[\ell]}$ such that $|D^*| = m.$ Therefore $D$ cannot be a minimum exponential dominating set. 
 
\item Suppose that $f^*( D ) = 1$ and $Z( D) \ge 1.$ By (\ref{condition1}) of Lemma \ref{newBC2}, there exists $2\ell \in V(C_{n,[\ell]})$ such that $w^*(D,2\ell) < \frac{6}{7}.$ Thus it is not possible for $2\ell$ to receive sufficient weight from $D,$ which implies that $D$ is not an exponential dominating set.  
\end{enumerate}
Assume that if $2\le f^*(D) + Z(D) < \alpha,$ then $D$ is not a minimum exponential dominating set. Now suppose $f^*(D) + Z(D)  = \alpha.$ We have the following three cases:
 
 \begin{enumerate}
 \item \label{c1} Suppose that $4 \le f^*(D ) \le 3\ell +1.$ Then by Remark \ref{f*Zeros} we have that $Z(D) \ge 3.$ With $D$ and $\I,$ we construct $D' \subset V(C_{n,[\ell]})$ using Algorithm \ref{shiftalg}. Then Lemma \ref{shiftalgproof} shows that $D'$ is an exponential dominating set such that $Z(D) = Z(D') -2,$ $f^*(D) -2 \le f^*(D') \le f^*(D),$ and $|D| = |D'|.$ Therefore $Z(D')  \ge 1$ and $2 \le f^*(D) -2\le f^*(D') \le f^*(D).$ This implies that $3 \le f^*(D') + Z(D') \le \alpha -1.$ By the induction hypothesis, $D'$ is not a minimum exponential dominating set. Thus $D$ cannot be a minimum exponential dominating set. 

\item \label{c2} Suppose that $f^*(D) = 3.$ Then by Remark \ref{f*Zeros}, $Z(D) \ge 2.$ With $D$ and $\I,$ construct $D' \subset V(C_{n,[\ell]})$ using Algorithm \ref{shiftalg}. Then Lemma \ref{shiftalgproof} shows that $D'$ is an exponential dominating set such that $Z(D) = Z(D') -2,$ $f^*(D) -2 \le f^*(D') \le f^*(D),$ and $|D| = |D'|.$ Therefore $Z(D')  \ge 0$ and $ 1\le f^*(D') \le 3.$ Consider the following three subcases:	
	\begin{enumerate}
	\item Consider the case when $f^*(D') \ge 1$ and $Z(D') \ge 1.$ Then we have that $ 2 \le f^*(D') + Z(D') \le \alpha -1.$ By the induction hypothesis, $D'$ is not a minimum exponential dominating set. Thus $D$ cannot be a minimum exponential dominating set.
	\item  Consider the case when $f^*(D') = 1$ and $Z(D' ) =0.$ This implies that there exists $I_i, I_j, I_k \subset \I$ for which $f_i(D) = f_k(D) = 0,$  $f_j(D) = 3,$ and $f_t(D) =1$ for all remaining $I_t \subset \I.$ By Lemma \ref{f*=3notDom}, $D$ is not an exponential dominating set.   
	\end{enumerate}

 \item \label{c3} Suppose that $f^*(D) = 2.$ Then $Z(D) \ge 1$ by Remark \ref{f*Zeros}. Without loss of generality we assume that the interval $I_0 = [0 , 3\ell]$ has $f_0(D) = 0.$ Lemma \ref{f_k=Z} show that either the intervals $I_1$ and $I_{m-1}$ have that $f_1(D) =2$ or $f_{m-1}(D) =2.$  Without loss of generality, suppose that $f_1(D) =2.$  Let  $I_1 \cap D = \{ d_0, d_1\}$ and let $d_i = I_i \cap D $ for all $2\le i \le m-1.$ Then consider the following two cases:
 
	 \begin{enumerate}
	 \item Suppose that $Z(D) \ge 2.$ By Lemma \ref{locations}, $d_0 , d_1 \in [3\ell+1,4\ell].$ Consider $D' =  (D \setminus \{d_0,d_1\} ) \cup \{  d_0 - \ell , d_1 + \ell  \}$ and Lemma \ref{2shift} shows that $D'$ is an exponential dominating set. By construction, we know that $d_0 - \ell \in I_0$ and $d_1 + \ell \in I_1.$ Therefore we have that $|D| = |D' |,$  $f_1(D') = 1,$ $f_0(D') = 1,$ and $f_t(D') = f_t(D)$ for all remaining $I_t \subset \I.$ This implies that $Z(D') \ge 1$ and $1\le f^*(D') \le 2.$ Then we have that $2 \le f^*(D') + Z(D') \le \alpha -1.$ By our induction hypothesis, $D'$ is not a minimum exponential dominating set. Thus, $D$ cannot be a minimum exponential dominating set.

	 \item Suppose that $Z(D) = 1.$ Then $f_i(D) = 1$ for $2\le i \le m-1.$ Observe that by Lemma \ref{locations} the minimum requirement on $d_0, d_1,d_{m-1}$ to ensure that $I_0$ gets exponentially dominated by $D$ is that $d_0, d_1 \in [4\ell -1, 4\ell]$ and $d_{m-1} = (3\ell +1)(m-1) + 2\ell +1.$ Through symmetry of the above argument, the minimum requirement to ensure that the interval $[ 4\ell + 1, 7\ell + 1]$ is exponentially dominated is that $d_2 = 8\ell + 1.$ Fix $j_0$ such that $3\le j_0 \le m-1,$ and suppose that the interval $[ (3\ell +1)(j_0-1) + 2\ell, (3\ell +1)j_0 + 2\ell - 1  ]$ contains no members of $D.$ Let $d \in \{ d_0, d_1 \}$ and notice that there is one member of $D\setminus d$ in every nonempty interval. Therefore $f^*(D\setminus d) = 1$ and $Z(D\setminus d) \ge 0.$ By (\ref{condition1}) of Lemma \ref{newBC2}, $w^*(D\setminus d , (3\ell+1)j_0  < \frac{6}{7}.$ This condition forces $w^*(d , (3\ell+1)j_0 ) > \frac{1}{7}.$ However $d \in I_1,$ so $w^*(d , (3\ell+1)j_0  ) < \frac{1}{7}.$ This gives that $w^*(D, (3\ell+1)j_0 ) < 1,$ a contradiction. Therefore the minimum requirement to ensure $[ (3\ell +1)(j_0-1) + 2\ell , (3\ell +1)j_0 + 2\ell -1 ]$ is exponentially dominated is that $d_{j_0} = (3\ell +1)j_0 + 2\ell -1.$  This implies $ (3\ell +1)(m-1) + 2\ell -1 , (3\ell +1)(m-1) + 2\ell  \in I_{m-1} \cap D,$ which contradicts that $f_{m-1}(D) = 1.$ Therefore $(3\ell +1)(m-1) + 2\ell -1 \not\in D$ and $w^*(D, (3\ell+1)(m-1) ) < 1.$ Thus it is not possible for $(3\ell+1)(m-2) +3\ell$ to receive sufficient weight from $D,$ which implies that $D$ is not an exponential dominating set.  

\end{enumerate}
\end{enumerate}
Through induction, it has been shown that if $f^*(D)+ Z(D) \ge 2,$ then $D$ is not a minimum exponential dominating set. Therefore if $D$ is an exponential dominating set, then $|D| =m $ such that $f^*(D) =1$ and $Z(D) = 0$ for all $3\ell +1$ distinct partitions $\I.$ What is left to show is that $D$ is unique up to isomorphism. Suppose that $0\in D$ and fix the remaining members of $D.$ Let $I_0 \in \I$ such that $I_0 = [0 , 3\ell].$ Therefore none of the remaining elements of $I_0$ can be members of $D.$ Shift the partition by one step to construct the interval $I'_0 =[ 1,  3\ell+1].$ Note that $| I'_0 \cap D| =1$ and $2, 3,$ $\ldots, 3\ell$ $\not \in D,$ so we must have $3\ell+1 \in D.$ Continuing this argument gives that $D = \{ (3\ell+1)k : 0 \le k \le m-1\}.$ Thus $D$ is unique up to isomorphism.
\end{proof}

\noindent \textit{Proof of Theorem \ref{mainCirculantThm}: } 
Let  $n = (3\ell+1)m + r$ and $D$ be a porous exponential dominating set for $C_{n,[\ell]}$ such that $|D| \le m.$ In the case where $r = 0,$ Theorem \ref{lmindom} shows that $D$ is a minimum porous exponential dominating set such that $|D| =m.$ Remark \ref{GenintDom} shows that $D$ forms a non-porous exponential dominating set. Thus using (\ref{PNPinequality}) we have that \[  \frac{n}{3\ell + 1}  \le \exd(C_{n,[\ell]}) \le \gamma_e(C_{n,[\ell]}) \le  \frac{n}{3\ell + 1}.\]
Consider the case when $r > 0.$ We first partition $H$ into $m+1$ intervals. Then notice that there must be at least one interval that contains no dominating vertices. We choose the partition $ \I = \displaystyle \cup_{i=0}^m I_i $  around $H$ such that $I_i = [ (3\ell + 1)i , (3\ell +1)i + 3\ell]$ for $0\le i \le m-1,$  $I_m = [(3\ell +1)m,$ $(3\ell +1)m + r - 1],$ and $f_{m+1}(D) = 0.$ Consider the graph $C_{n',[\ell]},$ where $n' = (3\ell+1)m.$ We define the vertex map  $\varphi : V(C_{n,[\ell]}) \to V(C_{n',[\ell]})$ such that $\varphi(i) = i $ for every $i \in \{0 ,1, \ldots, n' -1\}.$ Let $i,j \in V(C_{n,[\ell]}).$ As $\dist_H(\varphi(i), \varphi(j)) \le \dist_H(i,j),$ it follows that $D$ forms an exponential dominating set for $C_{n',[\ell]}.$ Theorem \ref{lmindom} shows that a minimum exponential domination set of for $C_{n',[\ell]}$ must have cardinality $m$ and is unique up to isomorphism. As $|D| \le m,$ $D$ must form a minimum exponential dominating set for $C_{n',[\ell]}$ with $|D| =m.$ Without loss of generality, let $D = \{   (3\ell +1)t : 0 \le t \le m-1\}.$ See Figure \ref{varphi} for an illustration of $D$ and the mapping $\varphi.$
 \begin{figure}[!htp]
\centering
\begin{tikzpicture}[scale = .95]
\node (z) at (-0.25, 0) {$|$};
\node (aa1) at (1.2,1) {$I_{m-2}$};
\node[circle, fill = white, draw] (a) at (0,0) {};
\node[circle, fill = white, draw] (b) at (0.40,0) {};
\node[circle, fill = white, draw] (c) at (.8,0) {};
\node[circle, fill = white, draw] (e) at (1.4,0) {};
\node[circle, fill = white, draw] (f) at (1.8,0) {};
\node[circle, fill = white, draw] (g) at (2.2,0) {};
\node (z1) at (2.5, 0) {$|$};
\draw[ thick, dotted ] (c) -- (e);

\node (aa1) at (3.9,1) {$I_{m-1}$};
\node[circle, fill = white, draw] (a1) at (2.8,0) {};
\node[circle, fill = white, draw] (b1) at (3.2,0) {};
\node[circle, fill = white, draw] (c1) at (3.6,0) {};
\node[circle, fill = white, draw] (d1) at (4.2,0) {};
\node[circle, fill = white, draw] (e1) at (4.6,0) {};
\node[circle, fill = white, draw] (f1) at (5,0) {};
\node (z3) at (5.5, 0) {$|$};
\draw[ thick, dotted ] (c1) -- (d1);

\draw[ thick, dotted ] (b1) -- (c1);
\node (aa1) at (6.7,1) {$I_{m}$};
\node[circle, fill = white, draw] (a2) at (5.8,0) {};
\node[circle, fill = white, draw] (c2) at (6.2,0) {};
\node[circle, fill = white, draw] (d2) at (6.8,0) {};
\node[circle, fill = white, draw] (f2) at (7.2,0) {};
\node[circle, fill = white, draw] (g2) at (7.6,0) {};
\node (z5) at (7.9, 0) {$|$};
\draw[ thick, dotted ] (c2) -- (d2);

\node (aa1) at (9.5,1) {$I_0$};
\node[circle, fill = white, draw] (a3) at (8.4,0) {};
\node[circle, fill = white, draw] (b3) at (8.8,0) {};
\node[circle, fill = white, draw] (c3) at (9.2,0) {};
\node[circle, fill = white, draw] (d3) at (9.8,0) {};
\node[circle, fill = white, draw] (f3) at (10.2,0) {};
\node[circle, fill = white, draw] (g3) at (10.6,0) {};
\node (z7) at (10.9, 0) {$|$};
\draw[ thick, dotted ] (c3) -- (d3);

\node (aa1) at (12.3,1) {$I_1$};
\node[circle, fill = white, draw] (a4) at (11.2,0) {};
\node[circle, fill = white, draw] (b4) at (11.6,0) {};
\node[circle, fill = white, draw] (c4) at (12,0) {};
\node[circle, fill = white, draw] (d4) at (12.6,0) {};
\node[circle, fill = white, draw] (f4) at (13,0) {};
\node[circle, fill = white, draw] (g4) at (13.4,0) {};
\node (z8) at (13.7, 0) {$|$};
\draw[ thick, dotted ] (c4) -- (d4);

\node[circle, fill = white, draw] (end1) at (14,0) {};

\node (start) at (-1.55,0){};
\node (aaaaa) at (-2.55,0){$C_{n,[\ell]}$};
\node[circle, fill = white, draw] (start1) at (-.55,0) {};

\node (end) at (15,0){};

\draw[ thick, dotted ] (start) -- (start1);
\draw[ thick, dotted ] (end1) -- (end);

\node (aa1) at (6.7,-0.5) {$\varphi$};

\node (z) at (-0.25, -1) {$|$};
\node[circle, fill = blue, draw] (1a) at (0,-1) {};
\node[circle, fill = white, draw] (1b) at (0.40,-1) {};
\node[circle, fill = white, draw] (1c) at (.8,-1) {};
\node[circle, fill = white, draw] (1e) at (1.4,-1) {};
\node[circle, fill = white, draw] (1f) at (1.8,-1) {};
\node[circle, fill = white, draw] (1g) at (2.2,-1) {};
\node (z1) at (2.5, -1) {$|$};

\draw[ thick, dotted ] (1c) -- (1e);

\node[circle, fill = blue, draw] (1a1) at (2.8,-1) {};
\node[circle, fill = white, draw] (1b1) at (3.2,-1) {};
\node[circle, fill = white, draw] (1c1) at (3.6,-1) {};
\node[circle, fill = white, draw] (1d1) at (4.2,-1) {};
\node[circle, fill = white, draw] (1e1) at (4.6,-1) {};
\node[circle, fill = white, draw] (1f1) at (5,-1) {};
\node (z3) at (6.6, -1) {$|$};
\draw[ thick, dotted ] (1c1) -- (1d1);

\draw[ thick, dotted ] (1b1) -- (1c1);

\node[circle, fill = blue, draw, label = below: \tiny $0$] (1a3) at (8.4,-1) {};
\node[circle, fill = white, draw] (1b3) at (8.8,-1) {};
\node[circle, fill = white, draw] (1c3) at (9.2,-1) {};
\node[circle, fill = white, draw] (1d3) at (9.8,-1) {};
\node[circle, fill = white, draw] (1f3) at (10.2,-1) {};
\node[circle, fill = white, draw, label = below: \tiny $3\ell$] (1g3) at (10.6,-1) {};
\node (z7) at (10.9, -1) {$|$};
\draw[ thick, dotted ] (1c3) -- (1d3);

\node[circle, fill = blue, draw] (1a4) at (11.2,-1) {};
\node[circle, fill = white, draw] (1b4) at (11.6,-1) {};
\node[circle, fill = white, draw] (1c4) at (12,-1) {};
\node[circle, fill = white, draw] (1d4) at (12.6,-1) {};
\node[circle, fill = white, draw] (1f4) at (13,-1) {};
\node[circle, fill = white, draw] (1g4) at (13.4,-1) {};
\node (z8) at (13.7, -1) {$|$};
\draw[ thick, dotted ] (1c4) -- (1d4);

\node[circle, fill = blue, draw] (1end1) at (14,-1) {};
\node (1start) at (-1.55,-1){};
\node (aaaaa) at (-2.55,-1){$C_{n',[\ell]}$};
\node[circle, fill = white, draw] (1start1) at (-.55,-1) {};
\node (1end) at (15,-1){};
\draw[ thick, dotted ] (1start) -- (1start1);
\draw[ thick, dotted ] (1end1) -- (1end);

\draw[ -> ] (a) -- (1a);
\draw[ -> ] (b) -- (1b);
\draw[ -> ] (c) -- (1c);
\draw[ -> ] (e) -- (1e);
\draw[ -> ] (f) -- (1f);
\draw[ -> ] (g) -- (1g);

\draw[ -> ] (a1) -- (1a1);
\draw[ -> ] (b1) -- (1b1);
\draw[ -> ] (c1) -- (1c1);
\draw[ -> ] (d1) -- (1d1);
\draw[ -> ] (e1) -- (1e1);
\draw[ -> ] (f1) -- (1f1);

\draw[ -> ] (a3) -- (1a3);
\draw[ -> ] (b3) -- (1b3);
\draw[ -> ] (c3) -- (1c3);
\draw[ -> ] (d3) -- (1d3);
\draw[ -> ] (f3) -- (1f3);
\draw[ -> ] (g3) -- (1g3);

\draw[ -> ] (a4) -- (1a4);
\draw[ -> ] (b4) -- (1b4);
\draw[ -> ] (c4) -- (1c4);
\draw[ -> ] (d4) -- (1d4);
\draw[ -> ] (f4) -- (1f4);
\draw[ -> ] (g4) -- (1g4);

\draw[ -> ] (start1) -- (1start1);
\draw[ -> ] (end1) -- (1end1);

\end{tikzpicture}

 \caption{Illustration of $\varphi,$ with edges removed and $D \subset V(C_{n',[\ell]})$ defined}
 \label{varphi}
 \end{figure}
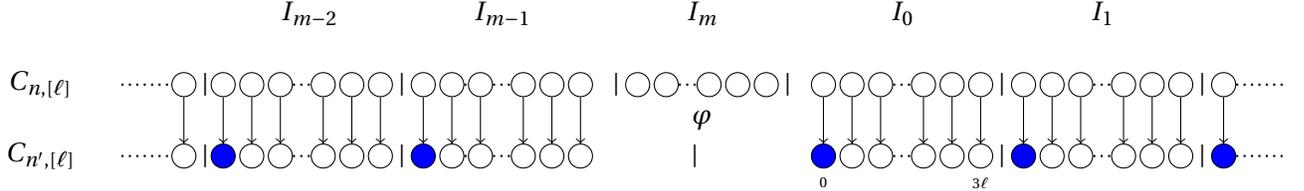 With regards to $C_{n,[\ell]},$ $D$ remains fixed since $I_m \cap D = \emptyset.$ Consider the intervals $I_0 = [0 , 3\ell ]$ and $I_1 = [3\ell +1, 6\ell +1 ].$ By construction, $0, 3\ell +1 \in D.$ Now shift the partition $\I$ so that $I_k = [ (3\ell +1)k + r+ 1, (3\ell+1)k + 3\ell + r + 1 \mod n]$ for $0\le k < m$ and $I_m = [(3\ell +1)m + r + 1 \mod n, (3\ell+1)m + 2r \mod n].$ Under $\varphi,$ notice that $\dist_H( \varphi(0) , \varphi(3\ell) ) = 3\ell + 1 - r.$ This shows that $D$ is not unique up to isomorphism in $C_{n',[\ell]},$ which contradicts Theorem \ref{lmindom}. Therefore $D$ cannot to be an exponential dominating set, see Figure \ref{varphi2} for an illustration of this contradiction. 
\begin{figure}[!htp]
\centering
\begin{tikzpicture}[scale = .95]

\node (z3) at (2.5, 0) {$|$};
\node (aa1) at (3.9,1) {$I_{m-1}$};
\node[circle, fill = white, draw] (a1) at (2.8,0) {};
\node[circle, fill = white, draw] (b1) at (3.2,0) {};
\node[circle, fill = white, draw] (c1) at (3.6,0) {};
\node[circle, fill = white, draw] (d1) at (4.2,0) {};
\node[circle, fill = white, draw] (e1) at (4.6,0) {};
\node[circle, fill = blue, draw] (f1) at (5,0) {};
\node (z3) at (5.5, 0) {$|$};
\draw[ thick, dotted ] (c1) -- (d1);

\draw[ thick, dotted ] (b1) -- (c1);
\node (aa1) at (6.7,1) {$I_{m}$};
\node[circle, fill = white, draw] (a2) at (5.8,0) {};
\node[circle, fill = white, draw] (c2) at (6.2,0) {};
\node[circle, fill = white, draw] (d2) at (6.8,0) {};
\node[circle, fill = white, draw] (f2) at (7.2,0) {};
\node[circle, fill = white, draw] (g2) at (7.6,0) {};
\node (z5) at (7.9, 0) {$|$};
\draw[ thick, dotted ] (c2) -- (d2);

\node (aa1) at (9.5,1) {$I_0$};
\node[circle, fill = white, draw] (a3) at (8.4,0) {};
\node[circle, fill = white, draw] (b3) at (8.8,0) {};
\node[circle, fill = white, draw] (c3) at (9.2,0) {};
\node[circle, fill = blue, draw] (d3) at (9.8,0) {};
\node[circle, fill = white, draw] (g3) at (10.6,0) {};
\node (z7) at (10.9, 0) {$|$};
\draw[ thick, dotted ] (c3) -- (d3)--(g3);

\node (aa1) at (12.3,1) {$I_1$};
\node[circle, fill = white, draw] (a4) at (11.2,0) {};
\node[circle, fill = white, draw] (b4) at (11.6,0) {};
\node[circle, fill = white, draw] (c4) at (12,0) {};
\node[circle, fill = blue, draw] (d4) at (12.6,0) {};
\node[circle, fill = white, draw] (f4) at (13,0) {};
\node[circle, fill = white, draw] (g4) at (13.4,0) {};
\node (z8) at (13.7, 0) {$|$};
\draw[ thick, dotted ] (c4) -- (d4);

\node (end) at (15,0){};
\node[circle, fill = white, draw] (end1) at (14,0) {};

\node (start) at (1.2,0){};
\node (aaaaa) at (0.5,0){$C_{n,[\ell]}$};
\node[circle, fill = white, draw] (start1) at (2.2,0) {};

\draw[ thick, dotted ] (start) -- (start1);
\draw[ thick, dotted ] (end1) -- (end);

\node (aa1) at (6.7,-0.5) {$\varphi$};

\node (z3) at (2.5, -1) {$|$};
\node[circle, fill = white, draw] (1a1) at (2.8,-1) {};
\node[circle, fill = white, draw] (1b1) at (3.2,-1) {};
\node[circle, fill = white, draw] (1c1) at (3.6,-1) {};
\node[circle, fill = white, draw] (1d1) at (4.2,-1) {};
\node[circle, fill = white, draw] (1e1) at (4.6,-1) {};
\node[circle, fill = blue, draw] (1f1) at (5,-1) {};
\node (z3) at (6.6, -1) {$|$};
\draw[ thick, dotted ] (1c1) -- (1d1);

\node[circle, fill = white, draw, label = below: \tiny $0$] (1a3) at (8.4,-1) {};
\node[circle, fill = white, draw] (1b3) at (8.8,-1) {};
\node[circle, fill = white, draw] (1c3) at (9.2,-1) {};
\node[circle, fill = blue, draw, label = below: \tiny $r$ ] (1d3) at (9.8,-1) {};
\node[circle, fill = white, draw, label = below: \tiny $3\ell$] (1g3) at (10.6,-1) {};
\node (z7) at (10.9, -1) {$|$};
\draw[ thick, dotted ] (1c3) -- (1d3) -- (1g3);

\node[circle, fill = white, draw] (1a4) at (11.2,-1) {};
\node[circle, fill = white, draw] (1b4) at (11.6,-1) {};
\node[circle, fill = white, draw] (1c4) at (12,-1) {};
\node[circle, fill = blue, draw] (1d4) at (12.6,-1) {};
\node[circle, fill = white, draw] (1f4) at (13,-1) {};
\node[circle, fill = white, draw] (1g4) at (13.4,-1) {};
\node (z8) at (13.7, -1) {$|$};
\draw[ thick, dotted ] (1c4) -- (1d4);

\node[circle, fill = white, draw] (1end1) at (14,-1) {};
\node (1start) at (1.2,-1){};
\node (aaaaa) at (0.5,-1){$C_{n',[\ell]}$};
\node[circle, fill = white, draw] (1start1) at (2.2,-1) {};
\node (1end) at (15,-1){};

\draw[ thick, dotted ] (1start) -- (1start1);
\draw[ thick, dotted ] (1end1) -- (1end);

\draw[ -> ] (a1) -- (1a1);
\draw[ -> ] (b1) -- (1b1);
\draw[ -> ] (c1) -- (1c1);
\draw[ -> ] (d1) -- (1d1);
\draw[ -> ] (e1) -- (1e1);
\draw[ -> ] (f1) -- (1f1);

\draw[ -> ] (a3) -- (1a3);
\draw[ -> ] (b3) -- (1b3);
\draw[ -> ] (c3) -- (1c3);
\draw[ -> ] (d3) -- (1d3);
\draw[ -> ] (g3) -- (1g3);

\draw[ -> ] (a4) -- (1a4);
\draw[ -> ] (b4) -- (1b4);
\draw[ -> ] (c4) -- (1c4);
\draw[ -> ] (d4) -- (1d4);
\draw[ -> ] (f4) -- (1f4);
\draw[ -> ] (g4) -- (1g4);
\end{tikzpicture}
 \caption{Illustration of why $D$ is not an exponential dominating set, with edges removed}
 \label{varphi2}
 \end{figure}
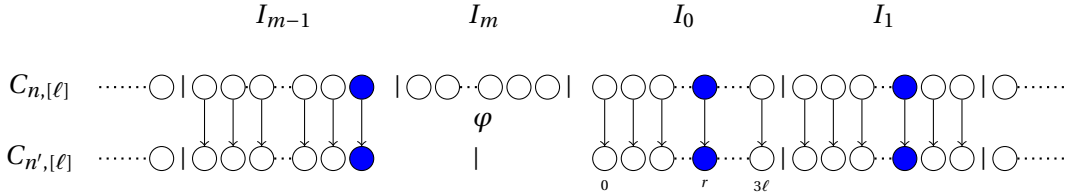
Consider $D' = D \cup v,$ where $v \in I_m.$ Observe that $\dist_H( d_k, d_{k+1}) \le 3\ell +1$ for consecutive $d_k, d_{k +1} \mod n \in D'.$  An application of Remark \ref{GenintDom} shows that $D'$ is a porous exponential dominating set for $C_{n,[\ell]}$ where $|D'| = m+1.$ Therefore $D'$ must be minimum. Additionally, Remark \ref{GenintDom} shows that $D'$ forms a non-porous exponential dominating set. Thus using (\ref{PNPinequality}) we have that \[ \left\lceil \frac{n}{3\ell + 1} \right\rceil \le \exd(C_{n,[\ell]}) \le \gamma_e(C_{n,[\ell]}) \le  \left\lceil \frac{n}{3\ell + 1} \right\rceil. \] \qed  
\end{subsection}
\end{section}

\section{Acknowledgements}
\label{sec:acknowledgements}
This research was supported in part by the National Science Foundation Award $\#$ 1719841. We would like to thanks Dr. Leslie Hogben for her input on this paper.

%%%%%%%%%%%%%%%%%%%%%%%%%%%%%%%%%%%%%%%%%%%%%%%%%%%%%%%%%%%%%%%%%%%%%%%%


\begin{thebibliography}{1}

\bibitem{anderson}
Mark Anderson, Robert~C. Brigham, Julie~R. Carrington, Richard~P. Vitray, and
  Jay Yellen.
\newblock On exponential domination of {$C_m\times C_n$}.
\newblock {\em AKCE Int. J. Graphs Comb.}, 6(3):341--351, 2009.

\bibitem{Ayta}
A.~Ayta\c{c} and B.~Atay.
\newblock On exponential domination of some graphs.
\newblock {\em Nonlinear Dyn. Syst. Theory}, 16(1):12--19, 2016.

\bibitem{Bessy1}
St{\'e}phane Bessy, Pascal Ochem, and Dieter Rautenbach.
\newblock Exponential domination in subcubic graphs.
\newblock {\em Electr. J. Comb.}, 23:P4.42, 2016.

\bibitem{Bessy}
St\'ephane Bessy, Pascal Ochem, and Dieter Rautenbach.
\newblock Bounds on the exponential domination number.
\newblock {\em Discrete Math.}, 340(3):494--503, 2017.

\bibitem{dankel}
Peter Dankelmann, David Day, David Erwin, Simon Mukwembi, and Henda Swart.
\newblock Domination with exponential decay.
\newblock {\em Discrete Math.}, 309(19):5877--5883, 2009.

\bibitem{Haynes}
Teresa~W. Haynes, Stephen~T. Hedetniemi, and Peter~J. Slater.
\newblock {\em Fundamentals of domination in graphs}, volume 208 of {\em
  Monographs and Textbooks in Pure and Applied Mathematics}.
\newblock Marcel Dekker, Inc., New York, 1998.

\bibitem{Henning1}
M.~A. {Henning}, S.~{J{\"a}ger}, and D.~{Rautenbach}.
\newblock {Hereditary Equality of Domination and Exponential Domination}.
\newblock {\em Discussiones Mathematicae Graph Theory}, 2016.

\bibitem{Henning}
Michael~A. Henning, Simon J\"ager, and Dieter Rautenbach.
\newblock Relating domination, exponential domination, and porous exponential
  domination.
\newblock {\em Discrete Optim.}, 23:81--92, 2017.

\end{thebibliography}
\end{document}